\documentclass[12pt]{article}
\usepackage{amssymb}
\usepackage{amsmath}

\usepackage{fancyhdr}

\usepackage{setspace}
\usepackage{enumitem}
\setlist{nolistsep}

\usepackage{graphicx}
\usepackage[usenames,dvipsnames]{color}
\usepackage{epsfig}

\usepackage[mathscr]{eucal}

\newtheorem{theorem}{Theorem}[section]
\newtheorem{lemma}[theorem]{Lemma}
\newtheorem{corollary}[theorem]{Corollary}

\newtheorem{definition}[theorem]{Definition}
\newtheorem{example}[theorem]{Example}
\newtheorem{remark}[theorem]{Remark}

\newtheorem{result}[theorem]{Result}
\newenvironment{proof}{\noindent{\bf Proof}\hspace{0.5em}}
    { \null  \hfill $\square$ \par}

\DeclareMathOperator{\Aut}{Aut} 
\DeclareMathOperator{\rad}{rad}

\newcommand{\X}{\mathcal X}
\newcommand{\Y}{\mathcal Y}

\newcommand\C{{\cal C}}
\newcommand\E{{\cal E}}

\newcommand\F{{\cal F}}

\renewcommand{\P}{\mathcal P}
\newcommand{\K}{\mathcal K}

\renewcommand{\H}{\mathcal H}
\newcommand{\Q}{\mathscr Q}

\renewcommand\setminus{\backslash}
\newcommand{\st}{:}

\newcommand\GF{{\rm GF}}
\newcommand\PG{{\rm PG}}

\newcommand\PGO{{\rm PGO}}

\newcommand{\Label}{\label}

\newcommand{\clique}{\mathcal C}
\newcommand{\setX}{{\mathcal C}_{{\rm\romannumeral 1}}}
\newcommand{\setY}{{\mathcal C}_{{\rm\romannumeral 2}}}
\newcommand{\setZ}{{\mathcal C}_{{\rm\romannumeral 3}}}

\newcommand\tthree{{\rm (\romannumeral 3)}}

\renewcommand{\Pr}{\P_{2r}}
\newcommand{\Er}{\E_{2r+1}}
\newcommand{\Hr}{\H_{2r+1}}

\begin{document}
%
%

\title{New families of strongly regular graphs}

\author{S.G. Barwick, Wen-Ai Jackson and Tim Penttila}
\date{}
\maketitle
Corresponding Author: Dr Susan Barwick, University of Adelaide, Adelaide
5005, Australia. Phone: +61 8 8313 3983, Fax: +61 8 8313 3696, email:
susan.barwick@adelaide.edu.au


\begin{abstract} In this article we construct a series of new infinite families of strongly regular graphs with the same parameters as the point-graphs of non-singular quadrics in $\PG(n,2)$. 
\end{abstract}
%

\section{Introduction}

A strongly regular graph srg$(v,k,\lambda,\mu)$, is a graph with  $v$ vertices such that
each vertex lies on $k$ edges;
any two adjacent vertices have exactly $\lambda$ common neighbours; and 
any two non-adjacent vertices have exactly $\mu$ common neighbours.
%
%
We consider the strongly regular graphs constructed  from a non-singular  quadric $\Q_n$ in $\PG(n,q)$. 
The {\em point-graph} $\Gamma_{\Q_n}$ of $\Q_n$ has vertices corresponding to the points of $\Q_n$. Two vertices in $\Gamma_{\Q_n}$  are adjacent  if the corresponding points of $\Q_n$ lie on a line contained in  $\Q_n$. It is well known (see for example \cite{brouwer89}) that $\Gamma_{\Q_n}$ is a strongly regular graph. In this article we let $q=2$, and construct from  $\Gamma_{\Q_n}$  approximately $n/2$ new strongly regular graphs with the same parameters as  $\Gamma_{\Q_n}$ (see Table~\ref{table3} for a precise count). 

This article proceeds as follows. Section~\ref{background} contains  several preliminary results we need. Section~\ref{sec:const} describes our construction of a series of infinite families of strongly regular graphs, the proof of the construction is given in  Section~\ref{sec:pf}. In Section~\ref{sec-cliques}, we classify and count the maximal cliques in the new graphs.
In  Section~\ref{sec-new}  we prove that our construction yields new families of strongly regular graphs. Finally, in Section~\ref{sec:autgp}, we determine the automorphism group of the new graphs.

In previous work, Kantor~\cite{kantor} 
constructed a strongly regular graph from $\Gamma_{\Q_n}$ with the same parameters in the case when $\Q_n$ contains a spread. Kantor conjects that his graph is not isomorphic to  $\Gamma_{\Q_n}$. 
We show in Section~\ref{sec:kantor} that the graph constructed by Kantor is not isomorphic to any of our new graphs. 
Abiad and Haemers~\cite{Abiad}  construct
several strongly regular graphs from the symplectic graph over $\GF(2)$. The dual of these graphs have the same parameters as the point-graph of a non-singular parabolic quadric, so $n$ is even.  It is not known if these graphs are isomorphic to our examples with $n$ even.

\section{Background Results}\Label{background}

In \cite{godsil}, Godsil and McKay take a graph $\Gamma$,  and use a vertex partition to  construct a new graph $\Gamma'$ that has the same spectrum as $\Gamma$. It is well-known (see for example \cite{brouwer}) that if a graph $\Gamma'$ has the same spectrum as a strongly regular graph $\Gamma$, then $\Gamma'$ is also strongly regular with the same parameters as $\Gamma$. Specialising the Godsil-McKay construction  to a partition of size two in a strongly regular graph gives the following result. 

\begin{result}\Label{GMK} 
\begin{enumerate} 
\item 
A {\em Godsil-McKay partition} of a graph is  a partition of the vertices into two sets $\{\X,\Y\}$ satisfying: 
\begin{itemize}
\item[{\rm I.}] The set $\X$ induces a regular subgraph. 
\item[{\rm II.}] Each vertex in $\Y$ is adjacent to $0$, $\frac12|\X|$ or $|\X|$ vertices in $\X$. 
\end{itemize}
\item {\em Godsil-McKay construction.} Let $\Gamma$ be a strongly regular graph with Godsil-McKay partition $\{\X,\Y\}$. Construct the graph $\Gamma'$ with the same points and edges as $\Gamma$, except: for each vertex $R$ in $\Y$ with $\frac12|\X|$ neighbours in $\X$, delete these $\frac12|\X|$ edges and join $R$ to the other $\frac12|\X|$ vertices in $\X$.
Then the graph $\Gamma'$ is strongly regular with the same parameters as $\Gamma$. 
\end{enumerate}
\end{result}

Let $\Q_n$ be a  non-singular quadric in $\PG(n,q)$.
The {\em projective index} $g$ of $\Q_n$ is the dimension of the largest subspace contained in $\Q_n$. A $g$-space contained in $\Q_n$ is called a {\em generator} of $\Q_n$.  If $n=2r$ is even, then a non-singular quadric is called a parabolic quadric, denoted  $\Pr$, which has projective index $g=r-1$.
If $n=2r+1$ is odd, then there are two types of non-singular quadrics: the elliptic quadric denoted $\Er$ has  projective index $g=r-1$; and the hyperbolic quadric denoted $\Hr$ has  projective index $g=r$. 
The points and generators of $\Q_n$ also form a  polar space of rank $g+1$. 
We repeatedly use the following two properties of quadrics and polar spaces, see 
\cite[Chapter 22]{HT} for more information on quadrics, and \cite[Section 26.1]{HT} for more information on polar spaces.

\begin{result} \Label{quadric-deg}
Let $\Q_n$ be a non-singular quadric in $\PG(n,q)$ and let $\Pi$ be a $k$-space. If the quadric $\Q_n\cap\Pi$ contains a $(k-1)$-space, then $\Q_n\cap\Pi$ is either $\Pi$, or one or two $(k-1)$-spaces.
\end{result}

\begin{result}\Label{property-star}
 Let $\Q_n$ be a non-singular quadric in $\PG(n,2)$, with projective index $g$. Let $\Sigma$ be  a generator of $\Q_n$, and $X$ a point of $\Q_n$ not in $\Sigma$. Then there is a unique generator $\Pi$ of $\Q_n$ that contains $X$ and meets $\Sigma$ in a $(g-1)$-space.  Further, the  points in $\Sigma$ which lie on a line of $\Q_n$ through  $X$ are exactly the points in $\Sigma\cap\Pi$.
\end{result}

\section{Our construction}\Label{sec:const}

We begin with a small example 
to illustrate the general construction. 

\begin{example}\Label{example} {\rm 
 Let $\ell$ be a line of the elliptic quadric  $\E=\Er$ in $\PG(2r+1,q)$. Partition the points of $\E$ into the following three types.
\begin{itemize}
\item[(i)] points of $\E$ on $\ell$, 
\item[(ii)] points of $\E$ that are on a plane of $\E$ that contains $\ell$, 
\item[(iii)] the  remaining points of $\E$.
\end{itemize}
Define a new graph $\Gamma_1$ with vertices the points of $\E$, and  edges given in  Table~\ref{table-edges}.
\begin{table}[h]
  \caption{Edges in $\Gamma_1$}\label{table-edges}
\begin{center}
\begin{tabular}{|c|l|l|}
\hline
Vertex pair & Vertex types &  Vertex pair  is an edge of $\Gamma_1$: \\
\hline
 $P,P'$  & $P,P'$ are type (i)&always (as $PP'$ is always a line of $\E$)\\
 $P,Q$ &  $P$ is type (i), $Q$ is type (ii)& always (as $PQ$ is always a line of $\E$)\\
 $Q,Q'$ & $Q,Q'$ are type (ii)&  when $QQ'$ is a line of $\E$\\
 $P,R$ & $P$ is type (i), $R$ is type (iii)&  when $PR$ is a line of $\E$\\
 $R,R'$ & $R,R'$ are type (iii)&   when $RR'$ is a line of $\E$\\
 $Q,R$ & $Q$ is type (ii), $R$ is type (iii)&  when  $QR$ is a $2$-secant of $\E$\\ \hline
 \end{tabular}
\end{center} \end{table}
Note that the last row of   Table~\ref{table-edges} describes the edges of $\Gamma_1$ that are different to the edges of the point-graph $\Gamma_\E$  of $\E$.

It can be shown directly using geometric techniques that $\Gamma_1$ is regular if and only if $q=2$, and that in this case $\Gamma_1$ is strongly regular with the same parameters as $\Gamma_\E$. 
This can also be proved using  the Godsil-McKay construction as follows.
Consider the partition $\{\X,\Y\}$ of $\Gamma_\E$ where $\X$ contains the vertices of type (ii), and $\Y$ contains the vertices of type (i) and (iii). Geometric techniques can be used to show that this partition satisfies the conditions of Result~\ref{GMK}(1) if and only if $q=2$. Note that the graph constructed in Result~\ref{GMK}(2) from this partition is the graph $\Gamma_1$, hence
 $\Gamma_1$  is strongly regular when $q=2$.   \hfill$\square$} \end{example}

We now give our general construction of a series of infinite families of strongly regular  graphs. This construction generalises Example~\ref{example}. First we define a partition of the vertices  of the point-graph of $\Q_n$.

\begin{definition}\Label{def-partition} 
Let $\Q_n$ be a non-singular quadric in $\PG(n,q)$, and let $\Gamma$ be the point-graph of $\Q_n$. 
Let $s$ be an integer with $0\leq s<g$, where $g$ is the projective index of $\Q_n$. 
Let $\alpha_s$ be an $s$-dimensional subspace  contained in ${\Q_n}$. 
The points of $\Q_n$  (and so the vertices of $\Gamma$) can be partitioned into three types:
\begin{itemize}
\item[ {\rm (i)}]  points in $\alpha_s$,
\item [{\rm (ii)}]  points of $\Q_n\setminus\alpha_s$ that lie in some $(s+1)$-dimensional subspace $\Pi$ with $\alpha_s\subset\Pi\subset\Q_n$, 
\item[{\rm (iii)}]  the remaining points of $\Q_n$.
\end{itemize}
Let $\X_s$ be the vertices of $\Gamma$ of type {\rm (ii)} and let $\Y_s$ be the vertices of $\Gamma$ of type {\rm (i)} and {\rm (iii)}.
\end{definition}
Note that if $s=g$, then there are no points of type (ii), so we  need $s<g$.
We will  show that the partition $\{\X_s,\Y_s\}$, $0\leq s<g$,  is a Godsil-McKay partition if and only if $q=2$. 
By \cite[Theorem 26.6.6]{HT}, the group fixing $\Q_n$ is transitive on the subspaces of dimension $s$ contained in $\Q_n$. So for each $s$,  $0\leq s<g$, we can use Result~\ref{GMK} to construct   a unique strongly regular graph $\Gamma_s$ from $\Gamma$. 
We state the main result here, and give the proof in Section~\ref{sec:pf}.

\begin{theorem}\Label{main-thm} In $\PG(n,2)$, let $\Q_n$ be a non-singular quadric of projective index $g\geq1$ with point-graph $\Gamma$. For each integer $s$, $0\leq s <g$,  
let $\Gamma_{s}$ be the graph obtained using the  Godsil-McKay construction with the partition $\{\X_s,\Y_s\}$
 defined in Definition~{\rm \ref{def-partition}}. Then $\Gamma_s$ is a strongly regular graph with the same parameters as $\Gamma$.
\end{theorem}

We show in Section~\ref{sec-new} that $\Gamma_0\cong\Gamma$, and that for each $n$,  $\Gamma_s$, $0\leq s<g$ are $g-1$ non-isomorphic graphs.

\section{Proof of Theorem~\ref{main-thm}}\Label{sec:pf}

%

Throughout this section, let $\Q_n$ be a  non-singular quadric in $\PG(n,q)$ of projective index $g$, and let $\alpha_s$ be a subspace of dimension $s$, $0\leq s<g$,  contained in $\Q_n$. Let $\Gamma$ be the point-graph of $\Q_n$, and let $\{\X_s,\Y_s\}$ be the partition of the vertices of $\Gamma$ (and so of the points of $\Q_n$) defined in Definition~\ref{def-partition}. We will show that $\{\X_s,\Y_s\}$ satisfies Conditions I and II of Result~\ref{GMK}. First we count the points in $\X_s$.

\begin{lemma}\Label{lemma-ell-1} 
\begin{enumerate}
\item If  $\Q_n=\Er$, then $|\X_s|=\displaystyle \frac{q^{s+1}(q^{r-s}+1)(q^{r-s-1}-1)}{(q-1)}$.\\[1mm]
\item If  $\Q_n=\Hr$,  then $|\X_s|=\displaystyle \frac{q^{s+1}(q^{r-s-1}+1)(q^{r-s}-1)}{(q-1)}$.\\[1mm]
\item If  $\Q_n=\Pr$, then  $|\X_s|=\displaystyle\frac{q^{s+1}(q^{r-s-1}+1)(q^{r-s-1}-1)}{(q-1)}$.
\end{enumerate}
\end{lemma}

\begin{proof}
We prove this in the case $\Q_n$ is $\E=\Er$, which has projective index $g=r-1$ and point-graph denoted $\Gamma_\E$. The cases when $\Q_n$ is $\Hr$ and $\Pr$ are proved in a very similar manner. 

 By \cite[Theorem 22.5.1]{HT},  the number of subspaces of dimension $s$ contained in $\E$ is 
 $$\frac{\Big((q^{r-s+1}+1)(q^{r-s+2}+1)\cdots(q^{r+1}+1)\Big)\times \Big((q^{r-s}-1)(q^{r-s+1}-1)\cdots(q^r-1)\Big)}{(q-1)(q^2-1)\cdots(q^{s+1}-1)}.$$ 
 Moreover, replacing `$s$' by  `$s+1$' in this equation gives  
 the number of subspaces of dimension $s+1$ contained in $\E$.
By \cite[Theorem 3.1]{hirs98}, the number of subspaces of dimension $s$ in a subspace of dimension $s+1$ is
$\big(q^{s+2}-1\big)/\big(q-1\big)$.
By \cite{HT}, the number of subspaces of dimension $s+1$ that contain $\alpha_s$ and are contained in $\E$ is a constant. To calculate it, we
 count  ordered pairs
 $(\Pi,\Sigma)$ where $\Pi$ is an $s$-dimensional subspace contained in $\E$, $\Sigma$ is an $(s+1)$-dimensional subspace contained in $\E$, and $\Pi\subset\Sigma$. 
 This count gives the number of subspaces of dimension $s+1$ that contain $\alpha_s$ and are contained in $\E$ is 
 \begin{equation}\label{eqn-1}
x=\frac{ (q^{r-s}+1)(q^{r-s-1}-1)}{(q-1)}. 
 \end{equation}
 Each of these subspace contains $q^{s+1}$ points that are not in $\alpha_s$. Hence $|\X_s|=xq^{s+1}$ as required.
\end{proof}

We now show that $\{\X_s,\Y_s\}$ satisfies Condition I of Result~\ref{GMK}. 

\begin{lemma}\Label{lemma-ell-2} 
Let $\Gamma^*$ be the subgraph of $\Gamma$ on the vertices in   $\X_s$. Then $\Gamma^*$ is a regular graph with degree $k$ where:
\begin{enumerate}
\item if $\Q_n=\Er$, then 
$\displaystyle k=(q^{s+1}-1)+\frac{q^{s+2}(q^{r-s-1}+1)(q^{r-s-2}-1)}{(q-1)}$;\\[1mm]
\item if $\Q_n=\Hr$, then 
$\displaystyle k=(q^{s+1}-1)+\frac{q^{s+2}(q^{r-s-2}+1)(q^{r-s-1}-1)}{(q-1)}$;\\[1mm]
\item if $\Q_n=\Pr$, then 
$\displaystyle k=(q^{s+1}-1)+\frac{q^{s+2}(q^{r-s-2}+1)(q^{r-s-2}-1)}{(q-1)}.$
\end{enumerate}
\end{lemma}

\begin{proof} We prove this in the case $\Q_n$ is $\E=\Er$, which has projective index $g=r-1$ and point-graph denoted $\Gamma_\E$. The cases when $\Q_n$ is $\Hr$ and $\Pr$ are proved in a very similar manner.

Let $Q$ be a vertex in $\X_s$, 
we need to count the number of vertices in $\X_s$ that are adjacent to $Q$. Recall that $\X_s$ consists of vertices of type (ii), so 
in $\PG(2r+1,q)$, 
$Q$ is a point of the quadric $\E$, and the $(s+1)$-dimensional space $\Sigma=\langle Q,\alpha_s\rangle$ is contained in $\E$.
A vertex $Q'$ in $\X_s$ is adjacent to $Q$ if the line $QQ'$ is  contained in $\E$. We  partition the lines of $\E$ through $Q$ into three families: $\F_1$ contains the lines of $\E$ through $Q$ that lie in $\Sigma$; 
$\F_2$ contains the lines of $\E$ through $Q$ (not in $\F_1$) that lie in an $(s+2)$-dimensional subspace that contains $\Sigma$ and is contained in $\E$; and $\F_3$ contains the remaining lines of $\E$ through $Q$.

We first look at $\F_1$. 
The number of lines in $\F_1$ equals the number of lines through a point in an $(s+1)$-dimensional subspace, so by \cite[Theorem 3.1]{hirs98}, 
\begin{equation}\label{eqn-2}
|\F_1|=\frac{(q^{s+1}-1)}{(q-1)}.
\end{equation}
 Each of the lines in $\F_1$ contains the point $Q$ and meets $\alpha_s$ in one point. So each line in $\F_1$ gives rise to $q-1$ vertices in $\X_s$ which are adjacent to $Q$ in the graph $\Gamma^*$. 
In total,  $\F_1$ contributes $(q-1)\times |\F_1|=(q^{s+1}-1)$ neighbours of $Q$ in $\Gamma^*$.

Next we look at $\F_2$. Replacing `$s$' by  `$s+1$' in (\ref{eqn-1}) gives the number of subspace of dimension $s+2$ that contain the $(s+1)$-space $\Sigma=\langle Q,\alpha_s\rangle$  and are contained in $\E$ is 
$(q^{r-s-1}+1)(q^{r-s-2}-1)/(q-1)$.  Similarly, (\ref{eqn-2}) can be generalised to show that the number of lines through $Q$ that lie in a subspace of dimension $s+2$, and do not lie in  the $(s+1)$-space $\Sigma$ is $\Big( (q^{s+2}-1)/(q-1)\Big) -\Big( (q^{s+1}-1)/(q-1)\Big) =q^{s+1}$. Hence $$|\F_2|=q^{s+1}\times \frac{(q^{r-s-1}+1)(q^{r-s-2}-1)}{(q-1)}.$$ Each line in $\F_2$ contains one point of $\Sigma$, and the remaining $q$ points correspond to  $q$ vertices that lie in $\X_s$ (and are not considered in $\F_1$). That is, each line in $\F_2$ contributes $q$ neighbours to $Q$ in the graph $\Gamma^*$. So in total, $\F_2$ contributes $q\times|\F_2|=q^{s+2}(q^{r-s-1}+1)(q^{r-s-2}-1)/(q-1)$ neighbours to $Q$ in the graph $\Gamma^*$. 

Finally we look at $\F_3$. 
Let $\ell$ be a line in $\F_3$, so $\ell$ contains $Q$, but the $(s+2)$-space $\Pi=\langle \alpha_s,\ell\rangle$ is not contained in $\E$. Suppose that $\ell$ contains another point $Q'$ that corresponds to a vertex  in $\X_s$. Then $\Pi\cap\E$ contains the two distinct $(s+1)$-dimensional subspaces 
$\Sigma=\langle \alpha_s,Q\rangle$  and $\Sigma'=\langle \alpha_s,Q'\rangle$. As $\Pi$ is not contained in $\E$, $\Pi$ meets $\E$ in exactly the two $(s+1)$-spaces $\Sigma$ and $\Sigma'$. Thus $\ell=QQ'$ is not a line of $\E$, and so $\ell$ contains exactly two points $Q,Q'$ that are  vertices of $\X_s$, moreover  they are not adjacent in $\Gamma^*$. Thus $\F_3$ contributes 0  neighbours to $Q$ in the graph $\Gamma^*$.

Summing the neighbours of $Q$ in $\Gamma^*$ obtained from the families $\F_1,\F_2,\F_3$ gives the required result. Note that if $s=g-1$, so $s=r-2$, then $|\F_2|=0$, and the degree of $\Gamma^*$ is $q^{r-1}-1$.
\end{proof}

Now we look at Condition II of Result~\ref{GMK}. Note that throughout the proofs in this article, we consistently use $P,P'$ to denote points of type (i); $Q,Q'$ to denote points of type (ii); and $R,R'$ to denote points of type (iii).

\begin{lemma}\Label{lemma-ell-3} 
The partition  $\{\X_s,\Y_s\}$ satisfies Condition II of Result~\ref{GMK} if and only if $q=2$.
\end{lemma}

\begin{proof}
We prove this in the case $\Q_n$ is $\E=\Er$, which has projective index $g=r-1$ and point-graph denoted $\Gamma_\E$. The cases when $\Q_n$ is $\Hr$ and $\Pr$ are proved in a very similar manner. 

We need to show that in the graph $\Gamma_\E$, each vertex in $\Y_s$ is adjacent to $0$, $\frac12|\X_s|$ or $|\X_s|$ vertices in $\X_s$.
There are two cases to consider since the vertices in $\Y_s$ are of type (i) or (iii). 
First consider a vertex $P$ in $\Y_s$ of type (i). Let $Q\in\X_s$, so $Q$ is a vertex of type (ii). Hence in $\PG(2r+1,q)$, $P\in\alpha_s$ and $Q$ lies in an $(s+1)$-space $\Pi$ with $\alpha_s\subset\Pi\subset\E$. Hence $PQ$ is a line of $\E$, and so $P$ and $Q$ are adjacent vertices in $\Gamma_\E$. That is, each vertex of type (i) in $\Y_s$ is adjacent to each of the $|\X_s|$ vertices in $\X_s$.

 Now consider a vertex $R$ in $\Y_s$ of type (iii). We count the number of vertices $Q$ in $\X_s$ for which $RQ$ is a line of $\E$. We will show that this number is not 0 or $|\X_s|$, and further, is $\frac12|\X_s|$ if and only if $q=2$. Let $\Sigma$ be a subspace of $\E$ of dimension $s+1$ that contains $\alpha_s$. So $\Sigma\setminus\alpha_s$ consists of points of type (ii), hence $R\notin\Sigma$. Consider the $(s+2)$-space $\Pi=\langle \Sigma,R\rangle$. As $\alpha_s\subset\Sigma$, we have $\langle\alpha_s,R\rangle\subset\Pi$. 
 As $R$ is of type (iii), $\langle\alpha_s,R\rangle$ is not contained in $\E$. Hence $\Pi$ is not contained in $\E$. 
 So $\Pi\cap\E$ contains the $(s+1)$-space $\Sigma$ and the point $R\notin\Sigma$. Hence by Result~\ref{quadric-deg},  
 $\Pi\cap\E$ is two distinct $(s+1)$-spaces. That is, $\Pi\cap\E=\{\Sigma,\Sigma'\}$ where $\Sigma'$ is an $(s+1)$-space that contains $R$. As  $R$ is type (iii),  $\Sigma'$ does not contain $\alpha_s$. 
 Hence $\Sigma\cap\Sigma'$ is an $s$-space distinct from $\alpha_s$, and so $\Sigma\cap\Sigma'\cap\alpha_s$ is a space of dimension $s-1$.
 Let $Q$ be a point in $\Sigma'\cap\Sigma$, $Q\notin\alpha_s$, so $Q$ has type (ii). As $Q,R\in\Sigma'\subset\E$, the line $m=QR$  is a line of $\E$.  
 
 Suppose the line $m=QR$ contains a second point $Q'$ of type (ii). So $\langle \alpha_s,Q'\rangle$ is an $(s+1)$-space contained in $\E$. Thus $\Pi$ contains three distinct $(s+1)$-spaces of $\E$, namely $\Sigma,\Sigma',\langle \alpha_s,Q'\rangle$, contradicting Result~\ref{quadric-deg}. Thus $m$
 contains exactly one point of type (ii), namely $Q$, and the rest of the points on $m$ are type (iii). Hence in the graph $\Gamma_\E$, the line $m$ gives rise to one neighbour of $R$ that lies in $\X_s$, namely $Q$. Thus each point of $\Sigma'\cap\Sigma$ not in $\alpha_s$ 
 gives rise to exactly one vertex in $\X_s$ that is a neighbour of $R$. 
 This is true for every $(s+1)$-space $\Sigma$ with $\alpha_s\subset\Sigma\subset\E$. Moreover, each neighbour  of $R$ in $\X_s$ corresponds to a point of $\E$ that lies in exactly one such $(s+1)$-space, so arises   exactly once in this way. 
 Hence the number of neighbours of $R$ that lie in $\X_s$ equals the number of points of $\E\setminus\alpha_s$ that lie in an $(s+1)$-space $\Sigma$ with $\alpha_s\subset\Sigma\subset\E$.
 We count these points.
 
 Firstly, the number of  $(s+1)$-dimensional spaces that contain $\alpha_s$ and are contained in $\E$ is given in  (\ref{eqn-1}). 
 Secondly, let $\Sigma$ be an $(s+1)$-space containing $\alpha_s$, and $\Sigma'$  an $(s+1)$-space that meets $\Sigma$ in an $s$-space not containing $\alpha_s$. Then the number of points in $\Sigma\cap\Sigma'$ which are not in $\alpha_s$ is $\big((q^{s+1}-1)/(q-1)\big)- \big((q^{s}-1)/(q-1)\big)
 =q^{s}$. Hence in the graph $\Gamma_\E$, there are 
\[
y=\frac{ q^{s}\ (q^{r-s}+1)(q^{r-s-1}-1)}{(q-1)}
\] vertices in $\X_s$ that are neighbours of $R$. 
To satisfy Condition II of Result~\ref{GMK}, we  need $y\in\{ 0, \ \frac12|\X_s|,\ |\X_s|\}$. 
Now $y=0$ if and only if $r-s-1=0$, which does not occur as $s<g=r-1$. Further, $|\X_s|$ is calculated in Lemma~\ref{lemma-ell-1}, and  $y<|\X_s|$. 
Using Lemma~\ref{lemma-ell-1}, 
 $y=|\X_s|/2$ if and only if $q=2$.
 
 Thus the vertices in $\Y_s$ of type (i) are adjacent to $|\X_s|$ of the vertices in $\X_s$. Further,  the vertices in $\Y_s$ of type (iii) are not adjacent to 0 or all the vertices of $\X_s$,  and are adjacent to  $\frac12|\X_s|$ of the vertices in $\X_s$ if and only if $q=2$. That is, Condition II of Result~\ref{GMK} is satisfied in for the partition $\{\X_s,\Y_s\}$ of $\Gamma_\E$ if and only if $q=2$.  
\end{proof}

It is now straightforward to prove Theorem~\ref{main-thm}.

{\bf Proof of Theorem~\ref{main-thm}}
Let $\Q_n$ be a non-singular quadric of $\PG(n,2)$ with projective index $g$.  Let $s$ be an integer with  $0\le s <g$,  let $\alpha_s$ be a $s$-space contained in $\Q_n$, and let  $\{\X_s,\Y_s\}$ be the partition given in Definition~\ref{def-partition}. By Lemmas~\ref{lemma-ell-2} and \ref{lemma-ell-3},  the partition $\{\X_s,\Y_s\}$ satisfies Conditions I and II of Result~\ref{GMK}(1). Hence we can use Result~\ref{GMK}(2) to construct a graph $\Gamma_s$. Note that as the group fixing $\Q_n$ is transitive on the $s$-spaces of $\Q_n$, $0\leq s\leq g$, different choices of the subspace $\alpha_s$ give rise to the same (up to isomorphism) graph.
So for any  $s$, $0\leq s<g$, the graph  $\Gamma_{s}$  is a strongly regular graph with the same parameters as $\Gamma$. \hfill $\square$

\begin{remark}
{\rm   As $0\leq s<g$, we have $g\geq1$. This places a bound on $n$:   when $\Q_n$ is a hyperbolic quadric,  we need $n\ge 3$; when $\Q_n$ is a parabolic quadric, we need $n\ge 4$; and when $\Q_n$ is an elliptic quadric, we need $n\ge 5$.
}\end{remark}

It is useful to note that the proof of Lemma~\ref{lemma-ell-3} gives a description of  the edges in the graph $\Gamma_s$. That is, let  $P,P'$ be vertices of type (i), $Q,Q'$ 
 vertices of type (ii), and
$R,R'$ 
 vertices of type (iii). Then $\{P,P'\}$, $\{P,Q\}$, $\{P,R\}$, $\{Q,Q'\}$, $\{R,R'\}$ are edges of $\Gamma_s$ if $PP'$, $PQ$, $PR$, $QQ'$, $RR'$ are lines of $\Q_n$ respectively; and $\{Q,R\}$ is an edge of $\Gamma_s$ if $QR$ is a 2-secant of $\Q_n$. In summary, we have:

\begin{corollary}\Label{cor-edges} Let $\Gamma_s$, $0\leq s<g$  be the graph constructed in Theorem~\ref{main-thm}. The adjacencies in $\Gamma_s$ are the same as those given in Table~\ref{table-edges}.  
\end{corollary}

%

\begin{remark}{\rm  We note that if $q\neq 2$, then geometric techniques similar to those used here show that the graph $\Gamma_s$ with $s>0$ is {\em not} regular. }
\end{remark}

\section{Maximal cliques of $\Gamma_s$}\Label{sec-cliques}

In Section~\ref{sec:max-descr},  we classify the maximal cliques in the  graph $\Gamma_s$, and in Section~\ref{sec:max-count},  we count them.

\subsection{Description of Maximal Cliques of $\Gamma_s$}\Label{sec:max-descr}

Throughout this section, let $\Q_n$ be a non-singular quadric of $\PG(n,2)$ of projective index $g$ with point-graph $\Gamma$. For $s$ an integer with $0\leq s<g$, let $\alpha_s$ be an $s$-space of $\Q_n$. Let  $\Gamma_s$ be the graph described in Theorem~\ref{main-thm}.

We first describe the maximal cliques of the   point-graph $\Gamma$ of $\Q_n$. 
The largest subspaces contained in $\Q_n$ are the generators, which have dimension $g$, and so contain
 $2^{g+1}-1$ points.
Further, any subspace of $\Q_n$ is contained in a generator of $\Q_n$. Hence the maximal cliques of $\Gamma$ have $2^{g+1}-1$ vertices and
correspond to generators of  $\Q_n$.
%
%

We want to study 
maximal cliques in $\Gamma_s$,  we begin by studying cliques of  $\Gamma_s$ of size $2^{g+1}-1$, then show that these are maximal. 
We define a {\em $g$-clique} of $\Gamma_s$ to be a clique of size $2^{g+1}-1$.
The next lemma describes two types of $g$-cliques of $\Gamma_s$, we show later that these are the maximal cliques of $\Gamma_s$. The first type 
 corresponds to generators  of $\Q_n$ containing $\alpha_s$, and so corresponds to maximal cliques of the original graph $\Gamma$.
Figure~\ref{fig-cliques} illustrates the two types of $g$-cliques described in
  Lemma~\ref{g-cliq-1}.

\begin{figure}[h]
\centering
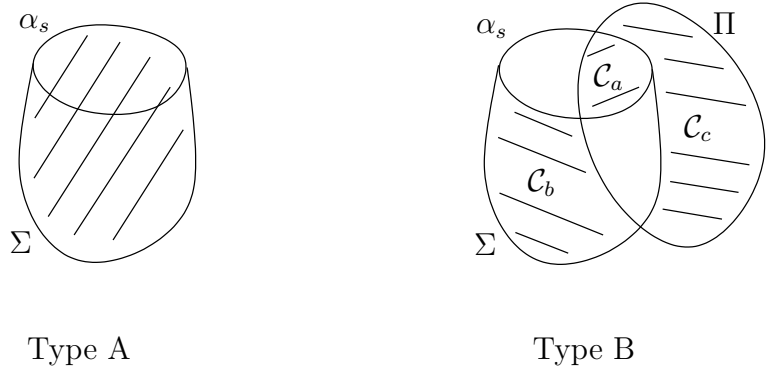
\caption{$g$-cliques of $\Gamma_s$}\label{fig-cliques}
\end{figure}

\begin{lemma}\Label{g-cliq-1} 
Let $\Gamma_s$, $0\le s < g$, be the graph constructed as in Theorem~\ref{main-thm}.
\begin{itemize}
\item[A.]
 Let $\Sigma$ be a generator of $\Q_n$ that contains $\alpha_s$, then the points of $\Sigma$ form a $g$-clique of $\Gamma_s$.
 \item[B.]  Let $\Pi,\Sigma$ be two generators of $\Q_n$ such that: $\Sigma$ contains $\alpha_s$; $\Pi$ does not contain $\alpha_s$; and $\Pi$, $\Sigma$ meet in a $(g-1)$-dimensional space. Let $\C_a$ be the $2^s-1$ points of $\alpha_s\cap\Pi$; $\C_b$ be the $2^g-2^s$ points of $\Sigma$ that are not in $\alpha_s$ or $\Pi$; and $\C_c$ be the $2^g$ points of $\Pi\setminus\Sigma$, see Figure~\ref{fig-cliques}. Then the points in $\C_a\cup\C_b\cup\C_c$ form a $g$-clique of the graph $\Gamma_s$.
 \end{itemize}
\end{lemma}

\begin{proof} 
For part A, let $\Sigma$ be a generator of $\Q_n$ that contains $\alpha_s$. 
Let $\C$ be the set of vertices of $\Gamma_s$ that correspond to the points of $\Sigma$. As  $\C$ consists of vertices of type (i) and (ii) only, two vertices of $\C$ are adjacent if the corresponding two points lie on a line of $\Q_n$. As $\Sigma$ is contained in $\Q_n$, every pair of distinct points in $\Sigma$ lie in a line of $\Q_n$. Hence every pair of distinct vertices in $\C$ are adjacent, so $\C$ is a clique. Further, $\Sigma$ contains $2^{g+1}-1$ points, so $|\C|=2^{g+1}-1$. Thus $\C$ is a $g$-clique of $\Gamma_s$. 

We now consider the set $\C_a\cup\C_b\cup\C_c$ described in part B.  
By construction, the three sets $\C_a, \C_b,\C_c$ are pairwise disjoint, $\C_a$ consists of points  of type (i),  
$\C_b$ consists of points of type (ii), and $\C_c$ contains no points of type (i). Suppose $\C_c$ contained a point $Q$ of type (ii), so $\langle \alpha_s,Q\rangle$ is an $(s+1)$-space of $\Q_n$. By construction, $\langle \alpha_s,Q\rangle$  is not contained in $\Pi$ or $\Sigma$, so contains a point $X$ not in $\Sigma$ or $\Pi$. So the $(g+1)$-space $\langle \Pi,\Sigma\rangle$ meets $\Q_n$ in at least $\Pi,\Sigma,X$, contradicting Result~\ref{quadric-deg}. Hence $\C_c$ consists of points of type (iii). Note that straightforward counting shows that the number of points in $\C_a, \C_b,\C_c$ is as stated in the theorem, and $|\C_a\cup\C_b\cup\C_c|=2^{g+1}-1$.

We need to show that any pair of vertices in the set corresponding to $\C_a\cup\C_b\cup\C_c$ are adjacent. Recall Corollary~\ref{cor-edges} shows that the adjacencies in $\Gamma_s$ are as described in Table~\ref{table-edges}. 
Let
 $P,P'\in\C_a$, $Q,Q'\in\C_b$, $R,R'\in\C_c$ be distinct points. (Note that the argument below is easily adjusted to work if $\C_a$ or $\C_b$ has size 1.) 
 As $P,P'$ have type (i), $Q,Q'$ have type (ii) and $R,R'$ have type (iii), the following pairs of points lie in a subspace of $\Q_n$, and so lie on a line of $\Q_n$:
 $P,P'\in\alpha_s\subset\Q_n$, 
 $Q,Q'\in\Sigma\subset\Q_n$, 
 $P,Q\in\Sigma\subset\Q_n$, 
 $P,R\in\Pi\subset\Q_n$,
 $R,R'\in\Pi\subset\Q_n$.
 Hence the corresponding pairs of vertices are all adjacent in $\Gamma_s$. 
 
 To complete the proof that 
 $\C_a\cup\C_b\cup\C_c$ corresponds to a $g$-clique of $\Gamma_s$, we need 
to show that $Q,R$ are adjacent in $\Gamma_s$, so by Table~\ref{table-edges}, we need to show that $QR$ is a 2-secant of $\Q_n$.
The line $QR$ lies in  the $(g+1)$-space $\langle\Pi,\Sigma\rangle$, which meets $\Q_n$ in exactly $\Pi$ and $\Sigma$. As $Q\in\Sigma\setminus\Pi$ and $R\in\Pi\setminus\Sigma$,  the line $QR$ is not contained in $\Q_n$, so it is a 2-secant of $\Q_n$. Hence $QR$ is an edge of $\Gamma_s$. That is, $\C_a\cup\C_b\cup\C_c$ is a set of $2^{g+1}-1$ vertices of $\Gamma_s$ such that any two vertices are adjacent, and so it is a $g$-clique of $\Gamma_s$. 
\end{proof}

We will show that  the only maximal cliques in $\Gamma_s$ are the $g$-cliques of Class A and B. We need some preliminary lemmas. 
Note that the $g$-cliques of Class A contain no points of type (iii), we begin by showing that  the converse also holds.

\begin{lemma}\Label{conv-g-cliq-1} Let $\C$ be a $g$-clique of $\Gamma_s$, $0\le s < g$, that contains no vertices of type \tthree, then $\C$ is a $g$-clique of Class A.
\end{lemma}

\begin{proof}
Let $\C$ be a $g$-clique of $\Gamma_s$, $0\le s < g$, that contains no vertices of type \tthree. Suppose $\C$ is not contained in a generator of $\Q_n$. We consider the number of points of $\C$ in each generator of $\Q_n$. Let $\Sigma$ be a generator of $\Q_n$ that contains the maximum number of points of $\C$. As $\C$ is not contained in $\Sigma$, there is a point $A$ of $\C$ that is not in $\Sigma$. By Result~\ref{property-star}, there is a unique generator $\Pi$ of $\Q_n$ that contains $A$ and meets $\Sigma$ in a $(g-1)$-space. Further, the points of $\Sigma$ that lie on a line of $\Q_n$ through $A$ are exactly the points of $\Sigma\cap\Pi$. 
 As $\C$ contains no points of type (iii), edges in $\C$ correspond to lines of $\Q_n$. In $\Gamma_s$, each vertex in $\C$ is adjacent to the vertex $A$, so in $\PG(n,2)$, the points of $\C\cap\Sigma$ lie in $\Sigma\cap\Pi$. Hence $|\Pi\cap\C|\geq |\Sigma\cap\C|+1$, which contradicts the choice of $\Sigma$ being the generator with the largest intersection with $\C$. Hence $\C$ is contained in a generator of $\Q_n$. As $|\C|=2^{g+1}-1$, the vertices of $\C$ correspond exactly to the points of this generator, and so $\C$ is a Class A $g$-clique.
\end{proof}

\begin{lemma}\Label{lem-1-type-ii} Every generator of $\Q_n$ contains at least one point of type {\rm (ii)}.
\end{lemma}
\begin{proof}
Let $\Q_n$ be a non-singular quadric of projective index $g$ and let $\Pi$ be a generator of  $\Q_n$. There are two cases to consider. Firstly, if $\Pi$ contains $\alpha_s$, then $\Pi$ contains only points of type (i) and (ii). Hence, as $s<g$, $\Pi$ contains at least one point of type (ii). 
Next consider the case where $\Pi$ meets $\alpha_s$ in a subspace $\alpha_t$ of dimension $t$,  with $-1\le t\le s-1$.  Let $P_1$ be a point of $\alpha_s\backslash\alpha_t$. As  $P_1\notin\Pi$, by Result~\ref{property-star} there exists a unique generator $\Sigma_1$ of $\Q_n$ that contains $P_1$ and meets $\Pi$ in a $(g-1)$-space. Moreover, if $Y\in\alpha_t$, then $P_1Y\subset\alpha_s$ and so is a line of $\Q_n$, hence by  Result~\ref{property-star}, $\alpha_t\subset \Sigma_1$, and so $\alpha_t= \Pi\cap\Sigma_1\cap\alpha_s$.
 Further, if $X$ is a point of  $\Pi\cap\Sigma_1$ not in $\alpha_s$ and $Y\in\langle\alpha_t,P_1\rangle$, then the line $XY$ lies in $\Sigma_1$ and so is a line of $\Q_n$. 

If $\alpha_s\cap\Sigma_1\neq\alpha_s$, we  repeat this process. Let  $P_2$ be a point of $\alpha_s$ not in $\Sigma_1$. By  Result~\ref{property-star} there is a generator  $\Sigma_2$ of $\Q_n$ that contains $P_2$ and meets $\Sigma_1$ in a $(g-1)$-space. Moreover, if $Y\in\langle \alpha_t,P_1\rangle$, then $P_2Y\subset\alpha_s$, and so  is a line of $\Q_n$, hence by  Result~\ref{property-star}, 
$\langle \alpha_t,P_1\rangle\subset\alpha_s\subset \Sigma_2$. So 
$\langle \alpha_t,P_1,P_2\rangle\subset\Sigma_2$, and $\alpha_t=\Pi\cap\Sigma_1\cap\Sigma_2\cap\alpha_s$.
Note that $\Pi\cap\Sigma_1\cap\Sigma_2$ has dimension at least $g-2$. 
Further, if $X$ is a point of $\Pi\cap\Sigma_1\cap\Sigma_2$ not in $\alpha_s$,
and $Y\in \langle \alpha_t,P_1,P_2\rangle$, 
 then $XY$ lies in $\Sigma_2$ and so is a line of $\Q_n$. 

Repeat this process a total of $k\leq	s-t$ times, 
until $\langle \alpha_t,P_1,\ldots,P_k\rangle=\alpha_s$.
Let $H=\Pi\cap\Sigma_1\cap\cdots\cap\Sigma_{k}$, so $H$ has dimension $d\geq g-k\geq  g-(s-t)$,  $H\cap\alpha_s=\alpha_t$, and $\alpha_s=\langle \alpha_t,P_1,\ldots,P_k\rangle\subset\Sigma_k$.
Note that $\dim H-\dim\alpha_t=d-t\geq g-(s-t)-t=g-s>0$, so $H\backslash \alpha_t$ is non-empty.
Let $X$ be a point of $H$ not in $\alpha_s$, and let $Y\in\alpha_s$. So $X,Y\in\Sigma_k$, hence   $XY$ is a line of $\Q_n$. That is,  $\langle X,\alpha_s\rangle$ is an $(s+1)$-space of $\Q_n$ and hence $X$ is a type (ii) point.  
As $X\in H\subset \Pi$, $\Pi$ contains at least one point of type (ii) as required.
\end{proof}

We now show that there are only two types of $g$-cliques in $\Gamma_s$, namely those of Class A and B described in Lemma~\ref{g-cliq-1}.

\begin{lemma}\Label{classify-cliq}
Let $\C$ be a $g$-clique in $\Gamma_s$, $0\leq s<g$,  then $\C$ is a $g$-clique of Class A or B.
\end{lemma}

\begin{proof}
Let $\clique$ be a $g$-clique of $\Gamma_s$ and denote the subsets of vertices of $\clique$ of type (i), (ii), (iii) by $\setX$, $\setY$, $\setZ$ respectively. 
If $\setZ=\emptyset$, then by Lemma~\ref{conv-g-cliq-1},
$\clique$ corresponds to a generator of $\Q_n$ containing $\alpha_s$, and so  is of Class A. So suppose $\setZ\neq \emptyset$. 

We begin by  constructing two generators of $\Q_n$ whose union contains the $g$-clique $\C$. 
Firstly, as $\C$ is a clique of $\Gamma_s$, the subset $\setX\cup\setZ$ is also a clique, so any two vertices of $\setX\cup\setZ$ are adjacent in $\Gamma_s$. As 
 $\setX\cup\setZ$ contains only vertices of type (i) and (iii), in $\PG(n,2)$, any two points of $\setX\cup\setZ$ lie on 
 a line of $\Q_n$. Hence $\setX\cup\setZ$ is a subspace of $\Q_n$ and so by \cite[Theorem 22.4.1]{HT} is contained in  a generator $\Pi$ of $\Q_n$. 
Secondly,  consider the set of points  $\alpha_s\cup\setY$ in $\Q_n$. Let $Q\in\setY$, so $Q$ has type (ii), and $\langle Q,\alpha_s\rangle$ is contained in $\Q_n$. Hence  $\alpha_s\cup\setY$ is a subspace of $\Q_n$ and so is contained in  a generator $\Sigma$ of $\Q_n$. So we have $\C\subset \Pi\cup\Sigma$. 
To show that $\C$ is a clique  of Class B, we need to show that $\Pi\cap\Sigma$ has dimension $g-1$.

We first show that $\setY$ is not empty. Suppose $\setY=\emptyset$, then 
$\C=\setX\cup\setZ$ is contained in the $g$-space $\Pi$. As $|\C|=2^{g+1}-1$, 
we have  $\C=\setX\cup\setZ=\Pi$.  However, by Lemma~\ref{lem-1-type-ii}, $\Pi$ contains at least one point of type (ii), a contradiction.  Thus $\setY\ne\emptyset$.

As $\setY,\setZ$ are not empty, let $Q\in\setY$ and $R\in\setZ$. As $Q,R$ lie in a clique of $\Gamma_s$, they are adjacent in $\Gamma_s$. Hence by Corollary~\ref{cor-edges},  $QR$ is a 2-secant of $\Q_n$. 
As $Q\in\setY\subset\Sigma\subset\Q_n$ and $QR$ is a 2-secant,  we have $R\notin\Sigma$. Similarly $R\in\setZ\subset\Pi\subset\Q_n$ and $QR$ a 2-secant implies $Q\notin\Pi$. 
 In summary,  we have 
 $$
  \C\subset\Sigma \cup \Pi;\quad \setX\subset\alpha_s\cap\Pi\cap\Sigma;\quad \setY\subset \Sigma\setminus\Pi; \quad \setZ\subset\Pi\setminus\Sigma.$$

Next we determine the size of $\setX$, $\setY$ and $\setZ$. 
As $\setZ\neq\emptyset$, there is a point $R\in\setZ$, so $R\notin\Sigma$. By Result~\ref{property-star},  there is a unique generator $\Pi_1$ of $\Q_n$ that contains $R$ and meets $\Sigma$ in a $(g-1)$-space denoted  $H=\Sigma\cap\Pi_1$. 
There are two cases to consider as $H\cap\alpha_s$ has dimension $s$ or $s-1$. If $H$ contained $\alpha_s$, then $\langle R,\alpha_s\rangle\subset\Pi_1$ would be a subspace of $\Q_n$, which implies that $R$ is type (ii), a contradiction. Thus $H\cap\alpha_s$ is an $(s-1)$-space. 
If $P\in\setX$, then $P,R\in\C$, so $P,R$ are adjacent in $\Gamma_s$ and so $PR$ is a line of $\Q_n$. Thus $P\in H$, and so $P\in H\cap\alpha_s$. Thus $\setX\subseteq H\cap\alpha_s$, and so 
 $|\setX|\leq |H\cap\alpha_s|=2^s-1$.  By the construction of $H$, each point in $H\setminus\alpha_s$ lies on a line of $\Q_n$ with $R$, and each point of $\Sigma\setminus(H\cup\alpha_s)$ lies on a 2-secant of $\Q_n$ with $R$. So the type (ii) points of $\C$ are contained in $ \Sigma\setminus(H\cup\alpha_s)$. That is, $|\setY|\leq|\Sigma\setminus(H\cup\alpha_s)|=(2^{g+1}-1)-\big((2^g-1)+2^s)\big)=2^g-2^s$.
 
As $\setY\neq\emptyset$, there is a point   $Q\in\setY$,  so $Q\in\Sigma\setminus\Pi$. 
 By Result~\ref{property-star},  there is a unique generator $\Sigma_1$ of $\Q_n$ that contains $Q$ and meets $\Pi$ in a $(g-1)$-space. Hence $Q$ is on a line of $\Q_n$ with the $2^g-1$ points of $\Pi\cap\Sigma_1$; and  $Q$ is on a 2-secant of $\Q_n$ with the $(2^{g+1}-1)-(2^g-1)=2^g$ points of $\Pi\setminus\Sigma_1$. 
If $R$ is a point of $\setZ$, then as $Q,R\in\C$, they are adjacent in $\Gamma_s$ and so $QR$ is a 2-secant of $\Q_n$. Hence the points of $\setZ$ lie in $\Pi\setminus\Sigma_1$, and so
 $|\setZ|\leq2^g$.

As $|\C|=2^{g+1}-1$, we need equality in all three of these  bounds, that is,
$|\setX|=2^s-1$,   
$|\setY|=2^g-2^s$,
and $|\setZ|=2^g$.
Moreover,  \begin{equation}\label{eqn-gcliq} \setX=\alpha_s\cap\Pi_1,\quad
\setY=\Sigma\setminus(\alpha_s\cup\Pi_1), \quad
\setZ=\Pi\setminus\Sigma_1.\end{equation}
To show that $\C$ is a $g$-clique of Class B, we need to show that $\Pi=\Pi_1$ and $\Sigma=\Sigma_1$. 
Suppose that $\Pi\ne\Pi_1$, so $\Pi\cap\Pi_1$ has dimension at most $g-1$, that is $|\Pi\cap\Pi_1|\leq 2^g-1$. As $\Pi$ contains $\setZ$,  and $|\setZ|=2^g>|\Pi\cap\Pi_1|$, 
there exists a point $R'\in\setZ$ with $R'\in\Pi\backslash\Pi_1$.  By Result~\ref{property-star}, there exists a unique generator $\Pi_2$ of $\Q_n$ which contains $R'$ and meets $\Sigma$ in a $(g-1)$-space. 
Further, for each point $X\in\Sigma\setminus\Pi_2$, $XR'$ is a 2-secant of $\Q_n$. Thus $\setY\subset\Sigma\setminus\Pi_2$. By (\ref{eqn-gcliq}), $\setY=\Sigma\setminus(\alpha_s\cup\Pi_1)$, moreover we have $|\Sigma\setminus(\alpha_s\cup\Pi_1)|=|\Sigma\setminus(\alpha_s\cup\Pi_2)|$. Hence $\Sigma\cap\Pi_1=\Sigma\cap\Pi_2$, and so $\Pi_1\cap\Pi_2$ is a $(g-1)$-space  in $\Sigma$. Recall that $R\in\Pi_1$, and by assumption $R'\in\Pi_2\setminus\Pi_1$, so $\Pi_1\neq\Pi_2$. Thus $\langle\Pi_1,\Pi_2\rangle$ is a $(g+1)$-space, and so by Result~\ref{quadric-deg}, meets $\Q_n$ in exactly the two generators $\Pi_1,\Pi_2$. Now $R,R'\in\setZ$, so $\{R,R'\}$ is an edge of $\Gamma_s$, and so $RR'$ is a line of $\Q_n$. As $R'\in\Pi_2\setminus\Pi_1$, and $RR'$ is a line of $\Q_n$ in $\langle\Pi_1,\Pi_2\rangle$, we have $R\in\Pi_2$. So $R\in\Pi_2\cap\Pi_1\subset\Sigma$, contradicting the choice of $R\not\in\Sigma$.
 Hence 
 $\Pi=\Pi_1$. Thus $\Sigma$ meets $\Pi$ in a $(g-1)$-space, so by the construction of $\Sigma_1$, we have $\Sigma=\Sigma_1$.
 Substituting into (\ref{eqn-gcliq}), we see that $\C$ is a $g$-clique of Class B.
\end{proof}

\begin{lemma}
 The maximum size of a clique in  $\Gamma_s$ is $2^{g+1}-1$.
 \end{lemma}

\begin{proof} 
Suppose $\Gamma_s$, $s>0$, contains a clique $\K$ of size $2^{g+1}$.  Let $X$ be a vertex in $\K$, then $\K\setminus X$ is a $g$-clique, and so by 
 Theorem~\ref{classify-cliq}, $\K\setminus X$ has  Class A or B.  Table~\ref{table-types} gives the number of vertices of each type in the two different $g$-cliques. 
 \begin{table}[h]\caption{Number of vertices of each type in each $g$-clique}\label{table-types}
\begin{center}
 \begin{tabular}{|c|c|c|}
\hline
&$g$-clique A&$g$-clique B\\
\hline
vertex type (i)&$2^{s+1}-1$&$2^s-1$\\ \hline
vertex type (ii)&$2^{g+1}-2^{s+1}$&$2^g-2^s$\\ \hline
vertex type (iii)&$0$&$2^g$\\ \hline
\end{tabular}\end{center}\end{table}
As $s>0$ and $\K\setminus X$ has  Class A or B,  $\K\setminus X$ contains vertices of both type (i) and (ii).
Let $P$ be a vertex of type (i) in $\K$ and $Q$ a vertex of type (ii) in $\K$. If $\K\setminus P$ has Class A, then using Table~\ref{table-types}, we see that $\K\setminus Q$ satisfies neither column, and so is not a $g$-clique of $\Gamma_s$, a contradiction. Similarly, if $\K\setminus P$ has Class B, then  $\K\setminus Q$ satisfies neither column, and so is not a $g$-clique of $\Gamma_s$. So there are no cliques of size $2^{g+1}$, hence the $g$-cliques are the maximal cliques of $\Gamma_s$. A similar argument proves the result when $s=0$. 
\end{proof}

In summary, we have classified the maximal cliques of $\Gamma_s$ as follows. 

\begin{theorem} \Label{max-cl}
Let $\Q_n$ be a non-singular quadric of $\PG(n,2)$ of projective index $g\geq1$, and let $\Gamma_s$, $0\leq s<g$, be the graph constructed in Theorem~\ref{main-thm}.
If $\C$ is a maximal clique of $\Gamma_s$, then $\C$ is a $g$-clique of Class A or B.
\end{theorem}

\subsection{Counting maximal cliques}\Label{sec:max-count}

In the previous section, we classified the maximal cliques in the  graph $\Gamma_s$, we count them here.

\begin{theorem}\Label{thm-max-cliq}
 Let $\Q_n$ be a non-singular quadric in $\PG(n,2)$ of projective index $g\geq1$. Let $\Gamma$ be the point-graph of $\Q_n$ and let $\Gamma_s$, $0\leq s<g$,
 be the graph constructed in Theorem~{\rm \ref{main-thm}}.
\begin{enumerate}
\item Let $\Q_n=\E_{2r+1}$, then 
\begin{enumerate}
\item  $\Gamma$ has $(2^2+1)(2^3+1)\cdots(2^{r+1}+1)$ maximal cliques.
\item $\Gamma_s$ has $(2^2+1)(2^3+1)\cdots(2^{r-s}+1)\times\big(2^{r+2}-2^{r-s+1}+1\big)$ maximal cliques.
\end{enumerate}
\item If $\Q_n=\H_{2r+1}$, then 
\begin{enumerate}
\item  $\Gamma$ has $(2^0+1)(2^1+1)\cdots(2^{r}+1)$ maximal cliques.
\item $\Gamma_s$ has $(2^0+1)(2^1+1))\cdots(2^{r-s-1}+1)\times\big(2^{r+1}-2^{r-s}+1\big)$ maximal cliques.
\end{enumerate}
\item If $\Q_n=\P_{2r}$, then 
\begin{enumerate}
\item  $\Gamma$ has $(2^1+1)(2^2+1)\cdots(2^{r}+1)$ maximal cliques.
\item $\Gamma_s$ has $(2^1+1)(2^2+1)\cdots(2^{r-s-1}+1)\times\big(2^{r+1}-2^{r-s}+1\big)$ 
maximal cliques.
\end{enumerate}
\end{enumerate}
\end{theorem}

\begin{proof} For part 1, we work in $\PG(2r+1,2)$ and let $\Q_n=\E=\E_{2r+1}$ have point-graph $\Gamma$.  
The maximal cliques of $\Gamma$ correspond exactly to the generators of $\E$. By \cite[Theorem 22.5.1]{HT},  the number  of generators of $\E$ is  $$(2^2+1)(2^3+1)\cdots(2^{r+1}+1)$$  proving 1(a).
For part 1(b), 
let $\alpha_s$ be a subspace of $\E$, $0\leq s<g$,  and let $\Gamma_s$ be the graph constructed from $\Gamma$  as in Theorem~\ref{main-thm}. Let $n_{\rm A}$, $n_{\rm B}$ be  the number of maximal cliques of $\Gamma_s$ of Class A and B respectively. 
By Lemma~\ref{g-cliq-1}, $n_{\rm A}$ is equal to the number of generators of $\E$ that contain $\alpha_s$, and so by \cite[Theorem 22.4.7]{HT}, 
\begin{equation}\label{eqn-nsA} n_{\rm A}=(2^2+1)(2^3+1)\cdots(2^{r-s}+1).\end{equation}
To count the maximal cliques of Class B, by Lemma~\ref{g-cliq-1}
we need to count the number of pairs of generators $\Sigma,\Pi$ of $\E$ such that $\Sigma$ contains $\alpha_s$, and $\Pi$ meets $\Sigma$ in a $(g-1)$-space not containing $\alpha_s$.
The number of choices for $\Sigma$ is the number of generators of $\E$ that contain $\alpha_s$, which is given in  (\ref{eqn-nsA}), and  is $n_{\rm A}$. Once $\Sigma$ is chosen, we count the number of choices for $\Pi$. The number of $(g-1)$-spaces contained in $\Sigma$ but not containing $\alpha_s$ equals the number of $(g-1)$-spaces contained in $\Sigma$ minus the number of $(g-1)$-spaces contained in $\Sigma$ which contain $\alpha_s$. This is $(2^{g+1}-1)-(2^{g-s}-1)=2^{g+1}-2^{g-s}$. By \cite[Lemma 22.4.8]{HT}, the number of generators of $\E$ that meet $\Sigma$ in a fixed $(g-1)$-space is four. Hence the number of choices for $\Pi$ is  $(2^{g+1}-2^{g-s})\times4=2^{g+3}-2^{g-s+2}$. As the projective index of $\E$ is $g=r-1$, we have   $n_{\rm B}=n_{\rm A}\big(2^{g+3}-2^{g-s+2}\big)= n_{\rm A}\big(2^{r+2}-2^{r-s+1}\big)$. Hence the total number of maximal cliques of $\Gamma_s$ is $n_{\rm A}+n_{\rm B}=
n_{\rm A}\big(2^{r+2}-2^{r-s+1}+1\big)$ as required. This completes the proof of part 1. The proofs of parts 2 and 3 are  similar.
\end{proof}

\begin{theorem}\Label{thm-max-cliq-thru-pt}  Let $\Q_n$ be a non-singular quadric in $\PG(n,2)$ of projective index $g\geq1$. Let $\Gamma_s$, $0\leq s<g$,
 be the graph constructed in Theorem~{\rm \ref{main-thm}}.
Let $X$ be a fixed vertex of $\Gamma_s$,  then  the number of maximal cliques of $\Gamma_s$ containing $X$ according to the type of $X$ is given in the next table. 

 {\rm
\begin{tabular}{c|c| l| l}
&&number of maximal cliques of $\Gamma_s$ containing $X$ \\

$\Q_n$& type of $X$& $0\leq s<g-1$&$s=g-1$\\ \hline

$\E_{2r+1}$&  {\rm (i)}& $(2^2+1)(2^3+1)\cdots(2^{r-s}+1)\times\big(2^{r+1}-2^{r-s+1}+1\big)$&$5(2^{r+1}-7)$\\
&  {\rm (ii)}&  $(2^2+1)(2^3+1)\cdots(2^{r-s-1}+1)\times\big(2^{r+1}-2^{r-s}+1\big)$&$2^{r+1}-3$\\
& {\rm  (iii)}&  $(2^2+1)(2^3+1)\cdots(2^{r-s}+1)$& 5\\
\hline

$\H_{2r+1}$& (i)&  $(2^0+1)(2^1+1)\cdots(2^{r-s-1}+1)\times\big(2^{r}-2^{r-s}+1\big)$& $2(2^r-1)$\\
& (ii)&  $(2^0+1)(2^1+1)\cdots(2^{r-s-2}+1)\times\big(2^{r}-2^{r-s-1}+1\big)$&$2^r$\\
& (iii)&  $(2^0+1)(2^1+1)\cdots(2^{r-s-1}+1)$&$2$
\\
\hline
$\P_{2r}$ &(i)&  $(2^1+1)(2^2+1)\cdots(2^{r-s-1}+1)\times\big(2^{r}-2^{r-s}+1\big)$&$3(2^r-3)$\\
& (ii)& $(2^1+1)(2^2+1)\cdots(2^{r-s-2}+1)\times\big(2^{r}-2^{r-s-1}+1\big)$&$2^r-1$\\
& (iii)&$(2^1+1)(2^2+1)\cdots(2^{r-s-1}+1)$&3\\
\end{tabular}
}\end{theorem}

\begin{proof}
First consider the case where  
$\Q_n=\E=\E_{2r+1}$
 in $\PG(n,2)=\PG(2r+1,2)$. Let $\alpha_s$ be a subspace of $\E$, $0\leq s<g$, and let $\Gamma_s$ be the graph constructed from the point-graph $\Gamma$ of $\E$,  as in Theorem~\ref{main-thm}.  Let $P$ be a vertex of $\Gamma_s$ of type (i), so in $\PG(2r+1,2)$, $P\in\alpha_s$. All the maximal cliques of $\Gamma_s$ of Class A
contain $\alpha_s$. So by (\ref{eqn-nsA}), $P$ lies in $n_{\rm A}=(2^2+1)(2^3+1)\cdots(2^{r-s}+1)$ maximal cliques of Class A. To form a maximal clique of $\Gamma_s$ of Class B that contains $P$, we need two generators $\Sigma,\Pi$ of $\E$ such that $\Sigma$ contains $\alpha_s$, $\Pi$ meets $\Sigma$ in a $(g-1)$-space not containing $\alpha_s$, and $P\in\Pi$. We count the number of pairs $\Sigma$, $\Pi$ satisfying this. First, the number of choices for $\Sigma$ equals the number of generators of $\E$ containing $\alpha_s$ which is $n_{\rm A}$. The number of $(g-1)$-spaces of $\Sigma$ that contain $P$ is $2^g-1$, and the number of $(g-1)$-spaces of $\Sigma$ that contain $\alpha_s$ and $P$ is $2^{g-s}-1$. Hence the number of $(g-1)$-spaces of $\Sigma$ that contain $P$, but do not contain $\alpha_s$ is $(2^g-1)-(2^{g-s}-1)=2^g-2^{g-s}$. By \cite[Lemma 22.4.8]{HT}, the number of generators of $\E$ that meet $\Sigma$ in a fixed $(g-1)$-space is four.  In total, the number of maximal cliques of Class B containing $P$ is $n_{\rm A}\times(2^g-2^{g-s})\times4=n_{\rm A}\big(2^{r+1}-2^{r-s+1}\big)$ as $\E$ has projective index $g=r-1$. Hence the total number of maximal cliques of $\Gamma_s$ containing 
$P$ is $n_{\rm A}\big(2^{r+1}-2^{r-s+1}+1\big)$ as required.

Now let $Q$ be   a vertex of $\Gamma_s$  of type (ii). The number of maximal cliques of Class A containing $Q$ equals the number of generators of $\E$ containing $\alpha_s$ and $Q$ which by \cite[Theorem 22.4.7]{HT} is  $(2^2+1)(2^3+1)\cdots(2^{r-s-1}+1)$. To count the maximal cliques of $\Gamma_s$ that contain $Q$, we need to count pairs of generators $\Sigma,\Pi$ of $\E$ such that $\Sigma$ contains $\alpha_s$ and $Q$, and $\Pi$ meets $\Sigma$ in a $(g-1)$-space not containing $\alpha_s$ or $Q$.
The number of choices for $\Sigma$ is calculated above to be $(2^2+1)(2^3+1)\cdots(2^{r-s-1}+1)$. Further, the number of $(g-1)$-spaces in $\Sigma$ is $2^{g+1}-1$; the number of $(g-1)$-spaces of $\Sigma$ containing $\alpha_s$ is $2^{g-s}-1$; 
the number of $(g-1)$-spaces of $\Sigma$ containing $\alpha_s$ and $Q$ is $2^{g-s-1}-1$; and 
the number of $(g-1)$-spaces of $\Sigma$ containing $Q$ is $2^{g}-1$.
Hence the number of $(g-1)$-spaces of $\Sigma$ that do not contain $\alpha_s$ and do not contain $Q$ is $(2^{g+1}-1)-(2^{g-s}-1)-(2^g-1)+(2^{g-s-1}-1)=2^g-2^{g-s-1}$. As before, each of these $(g-1)$-spaces lies in  four suitable choices for the generator $\Pi$ of $\E$. Hence the number of maximal cliques of Class B containing $Q$ is  $(2^2+1)(2^3+1)\cdots(2^{r-s-1}+1)\times (2^g-2^{g-s-1})\times 4 = (2^2+1)(2^3+1)\cdots(2^{r-s-1}+1)\big(2^{r+1}-2^{r-s}\big)$ as $\E$ has projective index $g=r-1$. Hence the total number of maximal cliques containing $Q$ is 
$(2^2+1)(2^3+1)\cdots(2^{r-s-1}+1)\big(2^{r+1}-2^{r-s}+1\big)$ as required.

Let $R$ be   a vertex of $\Gamma_s$ of type (iii), so $\langle R,\alpha_s\rangle$ is not contained in $\E$, hence $R$ is in zero maximal cliques of Class A. To count the maximal cliques of $\Gamma_s$ of Class B containing $R$, we  need to count pairs of generators $\Sigma,\Pi$ of $\E$ such that $\Sigma$ contains $\alpha_s$, $\Pi$ meets $\Sigma$ in a $(g-1)$-space not containing $\alpha_s$, and $\Pi$ contains $R$. The number of choices for $\Sigma$ equals the number of generators of $\E$ containing $\alpha_s$ which is $n_{\rm A}$ by (\ref{eqn-nsA}).  As $\Sigma$ contains $\alpha_s$, it contains no points of type (iii), so $R\notin\Sigma$. So by  Result~\ref{property-star}, there is a unique generator of $\E$ that contains $R$ and meets $\Sigma$ in a $(g-1)$-space denoted $H$. Further, if $H$ contained $\alpha_s$, then $\langle R,\alpha_s\rangle$ would be contained in $\E$, and so $R$ would be type (ii), a contradiction, so $H$ does not contain $\alpha_s$. So for each $\Sigma$, there is a unique choice for $\Pi$ that can be used to form a Class B maximal clique containing $R$. Hence the number of maximal cliques of $\Gamma_s$ containing $R$ is $n_{\rm A}=(2^2+1)(2^3+1)\cdots(2^{r-s}+1)$ as required.  This completes the proof for the case $\Q_n=\E_{2r+1}$. The cases when $\Q_n$ is $\H_{2r+1}$ and $\P_{2r}$ are    similar.
%
 \end{proof}

\section{The graphs $\Gamma_s$ are all  non-isomorphic}\Label{sec-new}

\begin{theorem}\Label{thm:s0iso}  Let $\Q_n$ be a non-singular quadric in $\PG(n,2)$ of projective index $g\geq1$. Let $\Gamma$ be the point-graph of $\Q_n$ and let $\Gamma_s$, $0\leq s<g$,
 be the graph constructed in Theorem~\ref{main-thm}.
 Then $\Gamma_s$ is isomorphic to $\Gamma$ if and only if $s=0$.
\end{theorem}
\begin{proof}
 We first show that $\Gamma_0\cong\Gamma$. 
To construct $\Gamma_0$ from $\Gamma$, we let $\alpha_0$ be a subspace of $\Q_n$ of dimension $0$, so $\alpha_0$ is a point which we denote $P$. We classify the points of $\Q_n$, and so the vertices of $\Gamma$, into type (i), (ii), (iii) with respect to $\alpha_0=P$. The point $P$ is the only point of $\Q_n$ of type (i). Note that lines in $\PG(n,2)$ contain exactly three points.  
Consider the involution $\phi$ acting on the vertices of $\Gamma$ where:  $\phi$ fixes vertices of type (i) and (iii); and $\phi$ maps a vertex $Q$ of type (ii) to the vertex of type (ii) that corresponds to the third point of $\Q_n$ on the line $PQ$. The involution $\phi$ maps $\Gamma$ to a graph $\Gamma'$.
 Incidence in $\Gamma'$ is inherited from $\Gamma$, that is, points $X$ and $Y$ are adjacent in $\Gamma$ (so $XY$ is a line of $\Q_n$) if and only if vertices $\phi(X)$ and $\phi(Y)$ are adjacent in $\Gamma'$. The map $\phi$ is an isomorphism, so $\Gamma\cong\Gamma'$. We 
now show that $\Gamma'=\Gamma_0$. 

By Corollary~\ref{cor-edges}, we need to show that the edges of $\Gamma'$ satisfy Table~\ref{table-edges}. 
First note that as there is only one point of type (i) in $\Q_n$, the first row of Table~\ref{table-edges} is not relevant.
Let  $Q_1,Q_2$ be points  of $\Q_n$ of type (ii), and let $R,R'$   be points  of $\Q_n$    of type (iii). 
The incidences in rows 4 and 5 of  Table~\ref{table-edges} hold in $\Gamma$, so as $\phi$ fixes points of type (i) and (iii), they also hold in $\Gamma'$. 

To simplify notation, let $Q_1'=\phi^{-1}(Q_1)$ and  $\phi^{-1}(Q_2)=Q_2'$. 
Consider row 2 of  Table~\ref{table-edges}: 
$\{P,Q_1\}$ is an edge of $\Gamma'$  if and only if $\{P,Q_1'\}$ is an edge of $\Gamma$ if and only if $\{P,Q_1,Q_1'\}$ is a line of $\Q_n$. Hence it follows from the definition of $\phi$ that $\{P,Q_1\}$ is always an edge of $\Gamma'$ as required. 

Consider row 6  of  Table~\ref{table-edges}:  
 $\{Q_1,R\}$ is an edge of $\Gamma'$  if $\{Q_1',R\}$ is an edge of $\Gamma$, that is, if 
$Q_1'R$ is a line of $\Q_n$. As $R$ is type (iii), the plane $\langle P,Q_1',R\rangle$ is not contained in $\Q_n$, and so by Result~\ref{quadric-deg} meets $\Q_n$ in exactly the lines $PQ_1'$, $Q_1'R$. 
As $Q_1$ is the third point on the line $PQ_1'$, the line  
 $Q_1R$ is a 2-secant of $\Q_n$ as required. 

Consider row 3  of  Table~\ref{table-edges}. 
Suppose $\{Q_1,Q_2\}$ is an edge of $\Gamma'$, so $\{Q_1',Q_2'\}$ is an edge of $\Gamma$. If the line $Q_1Q_2$ contains $P$, then $Q_1'=Q_2$ and $Q_2'=Q_1$, so $\{Q_1,Q_2\}$ is an edge of $\Gamma$ and so $Q_1Q_2$ is a line of $\Q_n$ as required. Now suppose $Q_1Q_2$ does not contain $P$. Then 
$\{Q_1',Q_2'\}$ an edge of $\Gamma$ implies $Q_1'Q_2'$ is a line of $\Q_n$. Hence the plane 
$\langle P,Q_1',Q_2'\rangle$ contains at least three lines, namely $PQ_1'$, $PQ_2'$ and $Q_1'Q_2'$, and so by Result~\ref{quadric-deg}, is contained in  $\Q_n$.  Further, it contains  $Q_1$ and $Q_2$, so $Q_1Q_2$ is a line of $\Q_n$ as required.
Hence the edges of $\Gamma'$ satisfy Table~\ref{table-edges}. So by Corollary~\ref{cor-edges}, $\Gamma'=\Gamma_0$.

We now show that $\Gamma_s$ with $1\leq s<g$ is not isomorphic to  the graph $\Gamma\cong\Gamma_0$ by considering the maximal cliques.
We prove the case when 
$\Q_n=\E=\E_{2r+1}$,  the cases where $\Q_n$ is $\H_{2r+1}$ or $\P_{2r}$ are similar.
The number of maximal cliques in $\Gamma$ and $\Gamma_s$ are given in (1a) and (1b) of Theorem~\ref{thm-max-cliq}. 
These numbers are equal if 
 and only if $2^{r+1}-2^{r-s+1}+1=(2^{r-s+1}+1)\cdots(2^{r}+1)$.
 If $s\geq 1$, then the right hand side is $\geq 2^{2r+1}$, which is larger than the left hand side.  So we have equality if and only if  $s=0$. 
Hence $\Gamma_s$ with $1\leq s<g$ is not isomorphic to $\Gamma$. \end{proof}

\begin{theorem}  Let $\Q_n$ be a non-singular quadric in $\PG(n,2)$ of projective index $g\geq1$. Let $\Gamma$ be the point-graph of $\Q_n$ and let $\Gamma_s$, $0\leq s<g$,
 be the graph constructed in Theorem~\ref{main-thm}.
Then the graphs
$\Gamma_0,\Gamma_1,\ldots,\Gamma_{g-1}$ are distinct up to isomorphism.
\end{theorem}

\begin{proof} We prove the case when 
$\Q_n=\E=\E_{2r+1}$,  the cases where $\Q_n$ is $\H_{2r+1}$ or $\P_{2r}$ are similar.
Let $s_1,s_2$ be two integers with $0\leq s_1<s_2<g$. The number of maximal cliques in $\Gamma_{s_1}$ and $\Gamma_{s_2}$ are given in Theorem~\ref{thm-max-cliq}(1b). These two numbers are equal if and only if 
\begin{eqnarray}
2^{r+2}-2^{r-s_2+1}+1&=&(2^{r-s_2+1}+1)\cdots(2^{r-s_1}+1)
\big( 2^{r+2}-2^{r-s_1+1}+1\big).\label{eqn:non-iso}
\end{eqnarray}
As $s_1<s_2$, the right hand side is greater than 
$2^{2r+2-s_1}$, which is greater than $2^{r+1}$ as  $s_1<s_2<g=r-1$. Hence the right hand side is greater than the left, so they cannot be equal. 
Thus $\Gamma_{s_1}$ and $\Gamma_{s_2}$ are not isomorphic  if  $s_1$ and $s_2$ are distinct. 
\end{proof}


\subsection{Kantor's graphs}\Label{sec:kantor}

In \cite{kantor}, Kantor constructs a strongly regular graph $\Gamma_K$ from a non-singular quadric $\Q_n$ in $\PG(n,q)$ with the same parameters as the point-graph $\Gamma$ of $\Q_n$. 
Kantor conjects that the graph $\Gamma_K$ is  not the same as $\Gamma$ except in the case when $\Q_n=\H_7$.  It is not known  in general whether $\Gamma_K$ is isomorphic to $\Gamma\cong\Gamma_0$.  We show that $\Gamma_K$ is not isomorphic to the graphs $\Gamma_s$ when $s>0$. Kantor's construction works when the quadric $\Q_n$ contains a spread, however, we do not need to describe the details of Kantor's graphs to prove non-isomorphism.  

\begin{theorem} Let $\Q_n$ be a non-singular quadric in $\PG(n,2)$ of projective index $g\geq1$. Let $\Gamma_s$, $0<s<g$ be the graph constructed in Theorem~\ref{main-thm}. Let $\Gamma_K$ be the graph constructed from $\Q_n$ in \cite{kantor}. Then $\Gamma_K$ is not isomorphic to $\Gamma_s$, $0<s<g$. \end{theorem}

\begin{proof} We use \cite[Lemma 3.3]{kantor} which shows that the vertices of $\Gamma_K$  can be partitioned into maximal cliques. We show that the vertices of $\Gamma_s$, $0<s<g$, cannot be partitioned into maximal cliques. 
Let $\C,\C'$ be two maximal cliques of $\Gamma_s$. 
We consider three cases.
If $\C,\C'$ are both of Class A, then they both contain $\alpha_s$, and so are not disjoint. If $\C$ is Class A and $\C'$ is Class B, 
then $\C$ contains $\alpha_s$, and $\C'$ meets $\alpha_s$ in a $(s-1)$-space. Hence as $s>0$, $\C'$ contains at least one point of $\alpha_s$, so $\C,\C'$ are not disjoint in this case.

Now consider the case where $\C,\C'$ are maximal cliques of $\Gamma_s$ of Class B. Both $\C,\C'$ meet $\alpha_s$ in a subspace of dimension $s-1$. If $s\geq 2$, then two subspaces of dimension $s-1$ contained in an $s$-space meet in at least a point,  and so $\C,\C'$ share at least a point.  Thus if $s\geq 2$, any two maximal cliques of $\Gamma_s$ share at least one vertex, and so the vertices of $\Gamma_s$  cannot be partition into maximal cliques, and hence $\Gamma_s$, $2\leq s<g$ is not isomorphic to $\Gamma_K$. 

Now suppose $s=1$, so $\alpha_1$ is a line. 
A partition of the vertices of $\Gamma_1$ into maximal cliques partitions the points of $\alpha_1$. As every maximal clique of $\Gamma_1$ contains a point of $\alpha_1$, we are looking for a partition of $\Gamma_1$ into 
 three maximal cliques of Class B, one through each point of $\alpha_1$.  We show there is no such partition. First, a maximal clique has $2^{g+1}-1$ points, so three pairwise disjoint maximal cliques contain $x=3(2^{g+1}-1)$ points, with either $g=r-1$ or $r$. 
As $0<s<g$, it follows that $g\ge2$. Thus for the elliptic and parabolic case we have $r\ge 3$ and for the hyperbolic case we have $r\ge 2$. 
However, as $q=2$, $\E_{2r+1}$ contains $2^{2r+1}-2^r-1$ points, $\H_{2r+1}$ contains $2^{2r+1}+2^r-1$ points and $\P_{2r}$ contains $2^{2r}-1$ points. 
None of these numbers is equal to $x$ when $r\ge 2$. Hence we cannot partition 
 the vertices of $\Gamma_s$, $s>0$ into maximal cliques. 
Thus by \cite[Lemma 3.3]{kantor}, $\Gamma_s$ is not isomorphic to $\Gamma_K$. 
\end{proof}

\section{The automorphism group of $\Gamma_s$}\Label{sec:autgp}

Let $\Q_n$ be a non-singular quadric in $\PG(n,2)$ of projective index $g\geq 1$. Let $\Gamma$ be the point-graph of $\Q_n$.
Let $\alpha_s$ be an $s$-space contained in $\Q_n$, $0\leq s<g$,  construct the partition of the points of $\Q_n$  given in Definition~\ref{def-partition}, and  let $\Gamma_s$ be the graph constructed in Theorem~\ref{main-thm}. 
If $s=0$, then by Theorem~\ref{thm:s0iso}, $\Gamma_0=\Gamma$  so $\Aut(\Gamma_0)=\Aut\Gamma$. 
In this section we determine  the automorphism group of the graph $\Gamma_s$, $0< s<g$.

First note that the group of collineations of $\PG(n,2)$ fixing $\Q_n$ is $\PGO(n+1,2)$, see \cite{HT}. Moreover, if $n\geq 4$, then the group of automorphisms of $\Gamma$ is $\Aut \Gamma\cong \PGO(n+1,2)$, see \cite[Chapter 8]{tits74}.

The partition of the points of $\Q_n$  given in Definition~\ref{def-partition} also partitions the vertices of $\Gamma$ and $\Gamma_s$, $0\leq s<g$. 
Vertices of type (i) in $\Gamma$ correspond in $\PG(n,2)$  to the points of $\alpha_s$. Let $(\Aut\Gamma)_{\alpha_s}$ denote the subgroup of automorphisms of $\Gamma$ that fix the set of vertices of type (i). 
As the graphs $\Gamma,\Gamma_s$ have the same set of vertices, if $\phi$ is a map acting on the vertices of $\Gamma$, then $\phi$ is also a map acting on the vertices of $\Gamma_s$. We will prove the following relationship between their automorphism groups. 

\begin{theorem}\Label{thm-gp}
Let $\Q_n$ be a non-singular quadric in $\PG(n,2)$ of projective index $g\geq 1$ with point-graph $\Gamma$. Let $\alpha_s$ be an $s$-space of $\Q_n$, $0< s<g$, and let $\Gamma_s$ be the graph constructed in Theorem~\ref{main-thm}. Then 
$\Aut(\Gamma_s)=(\Aut\Gamma)_{\alpha_s}$. 
\end{theorem}

In order to prove this theorem, we need a series of preliminary lemmas, the first relies on an application of Witt's Theorem, so we begin with a discussion on applying  Witt's Theorem to non-singular quadrics of $\PG(n,2)$, see 
\cite[Chapter 7]{taylor1992geometry} for more details. Let $V$ be a vector space of dimension $n+1$ over $\GF(2)$, and let $f(x_0,\ldots,x_n)$ be a quadratic form on $V$ with associated bilinear form $b(x,y)=b(x+y)-b(x)-b(y)$. The radical of $f$ in $V$  is the subspace $\rad f=\{u\in V\st b(u,v)=0 {\rm \ for\ all\ } v\in V\}$. 
Let $U$ be a subspace of $V$ and suppose there exists a linear isometry $\varphi\colon U\rightarrow V$  with respect to $f$ (that is, $\varphi$ is an invertible linear map and $f(u)=f(\varphi(u))$ for all $u\in U$. 
Then Witt's theorem says that  there exists a linear isometry $\zeta\colon V\rightarrow V$ such that $\zeta(u)=\varphi(u)$ for all $u\in U$ if and only if $\varphi(U\cap\rad f)=\varphi(U)\cap \rad f$. We interpret this in the projective space $\PG(n,2)$ associated with $V$. Let $\Q_n$ be a non-singular quadric in $\PG(n,2)$ with homogeneous equation $f(x_0,\ldots,x_n)=0$. If $n$ is odd, then $\rad f=\emptyset$. If $n$ is even, then $\Q_n=\P_n$ and $\rad f$ is the nucleus point $N$ of $\P_n$. 
As an example, let  $\Pi_1,\Pi_2$ be subspaces of $\Q_n$ of the same dimension. If $\Q_n$ has a nucleus $N$, then $N\notin\Q_n$, so neither $\Pi_1$ nor $\Pi_2$ contain $N$. 
As there exists a collineation   of $\PG(n,2)$ that maps $\Pi_1$ to $\Pi_2$,  it follows from Witt's theorem that  there exists a collineation of $\PG(n,2)$ that fixes $\Q_n$ and maps $\Pi_1$ to $\Pi_2$.
We use Witt's Theorem to prove the following lemma.

\begin{lemma}\Label{lem:three-orbs} Let $\Q_n$ be a non-singular quadric in $\PG(n,2)$  of projective index $g\geq 1$. Let $s$ be an integer,  $0\leq s<g$, let $\alpha_s$ be an $s$-space of $\Q_n$, and partition the points of $\Q_n$ into types {\rm (i), (ii), (iii)} as in Definition~\ref{def-partition}. Then the subgroup of $\PGO(n+1,2)$ fixing $\alpha_s$ is transitive on the points of each type.
\end{lemma}

\begin{proof}
Let $P,P'$ be two points of $\Q_n$ of type (i), so $P,P'\in\alpha_s$. There is a collineation of $\PG(n,2)$ that fixes $\alpha_s$, and maps $P$ to $P'$. Hence by Witt's theorem, there is a collineation of $\PG(n,2)$ fixing $\alpha_s$ and $\Q_n$, and mapping $P$ to $P'$. Hence $\PGO(n+1,2)_{\alpha_s}$ is transitive on the points of $\Q_n$ of type (i).

Let $Q,Q'$ be points of $\Q_n$ of type (ii), so $\Pi=\langle Q,\alpha_s\rangle$ and  $\Pi'=\langle Q',\alpha_s\rangle$ are $(s+1)$-spaces contained in $\Q_n$. There is a collineation of $\PG(n,2)$ that maps $\Pi$ to $\Pi'$,  fixes $\alpha_s$, and maps $Q$ to $Q'$. Hence by Witt's Theorem, there is a collineation of $\PG(n,2)$ that fixes  $\alpha_s$ and $\Q_n$, and maps $Q$ to $Q'$. Hence $\PGO(n+1,2)_{\alpha_s}$ is transitive on the points of $\Q_n$ of type (ii).

Let $R,R'$ be points of $\Q_n$ of type (iii), so $\Pi=\langle R,\alpha_s\rangle$ and  $\Pi'=\langle R',\alpha_s\rangle$ are $(s+1)$-spaces which are not contained in $\Q_n$. 
Now $\Pi$ is an $(s+1)$-space, and $\Pi\cap\Q_n$ contains $\alpha_s$ and the point $R\notin\alpha_s$, hence by Result~\ref{quadric-deg}, $\Pi\cap\Q_n$ is exactly two $s$-spaces. Similarly, $\Pi'\cap\Q_n$ is two $s$-spaces, one being $\alpha_s$. 
So there is an automorphism of $\PG(n,2)$ that maps $\Pi$ to $\Pi'$, fixes $\alpha_s$, and maps $R$ to $R'$. As $\Pi,\Pi'$ are  not contained in $\Q_n$, in order to apply Witt's Theorem, we need to consider the nucleus $N$ of $\Q_n$ when $n$ is even. 
Suppose $n$ is even, so $\Q_n=\P_n$, and $\P_n$ has nucleus a point $N\notin\P_n$. We show that neither $\Pi$ nor $\Pi'$ contain $N$.
Let $P\in\alpha_s\subset\P_n$ and let $\Sigma_P$ be the tangent hyperplane to $\P_n$ at $P$. So $\Sigma_P$ contains $N$ and all the lines of $\P_n$ through $P$. Let $\Sigma=\cap_{P\in\alpha_s} \Sigma_P$, then $\Sigma$ contains $N$ and points of type (i) and (ii), but no points of type (iii). 
As $\langle \alpha_s,N\rangle$ is an $(s+1)$-space contained in $\Sigma$, it contains no points of type (iii). As the $(s+1)$-space $\Pi$ contains points of type (iii), $\Pi$ meets $\langle \alpha_s,N\rangle$ in exactly the $s$-space $\alpha_s$. Thus $N\notin \Pi$, and similarly $N\notin\Pi'$. 
Hence by Witt's Theorem there is a collineation of $\PG(n,2)$ that fixes $\alpha_s$ and $\Q_n$, and maps $R$ to $R'$. Thus  $\PGO(n+1,2)_{\alpha_s}$ is transitive on the points of $\Q_n$ of type (iii).
\end{proof}

We now show that if $s>0$, then $\Aut(\Gamma_s)$  has at least three orbits on the vertices of $\Gamma_s$, namely the vertices of each type.

\begin{lemma}\Label{lem:diffcliq} 
For $0<s<g$, the vertices of  $\Gamma_s$ of different types lie in a different number of maximal cliques. 
\end{lemma}

\begin{proof} 
We prove the result for the case $\Q_n=\E_{2r+1}$, the cases when $\Q_n$ is $\H_{2r+1}$ and $\P_{2r}$ are similar.
Comparing the number of cliques through vertices of type (i), (ii) and (iii) in $\Gamma_s$ from Theorem~\ref{thm-max-cliq-thru-pt}, it is sufficient to show that $k_!,k_2,k_3$ are distinct where 
\[
k_1=(2^{r-s}+1)(2^{r+1}-2^{r-s+1}+1),\quad k_2 =2^{r+1}-2^{r-s}+1
,\quad  k_3=2^{r-s}+1.
\]
If $0<s<g-1$, then $k_1-k_2=2^{r-s}(2^{r+1}-2^{r-s+1})>0$ and  $k_2-k_3=2^{r+1}-2^{r-s+1}>0$. Hence $k_1>k_2>k_3$, that is vertices of different types lie in a different number of maximal cliques.  If $0<s=g-1$, then $r\geq 3$ and so $k_1,k_2,k_3$ are distinct. 

\end{proof}

\begin{lemma}\Label{thm:3-orbit} 
If $0<s<g$, then $(\Aut \Gamma)_{\alpha_s} \subseteq \Aut(\Gamma_s)$. Further, $\Aut(\Gamma_s)$ has exactly three orbits on the vertices of $\Gamma_s$, namely the vertices of each type.
\end{lemma}

\begin{proof} Recall that $\Gamma$ is the point-graph of a non-singular quadric  $\Q_n$ in $\PG(n,2)$ with projective index $g\geq 1$;  $\alpha_s$ is an $s$-space contained in $\Q_n$; and the vertices of $\Gamma$ are partitioned into  types (i), (ii) and (iii) as given in Definition~\ref{def-partition}.
Note that as $0<s<g$, we need $g>1$, and so $n\geq 5$.

 Let $\phi\in(\Aut\Gamma)_{\alpha_s}$, so $\phi$ is an automorphism of $\Gamma$ that fixes  the set of vertices of $\Gamma$ of type (i).  As $\Gamma,\Gamma_s$ have the same set of vertices, $\phi$ acts on the vertices of $\Gamma_s$, and fixes the set of vertices of $\Gamma_s$ of type (i). 
Further,  $\phi$ induces a bijection denoted $\bar{\phi}$ acting on the points of $\Q_n$ and fixing $\alpha_s$. As $n\geq 5$, we have $\Aut \Gamma\cong \PGO(n+1,2)$ (see \cite[Chapter 8]{tits74}) so $\bar{\phi}\in\PGO(n+1,2)_{\alpha_s}$. By Lemma~\ref{lem:three-orbs},  $\bar\phi$ preserves the type of a point in $\Q_n$, hence $\phi$ preserves the type of a vertex in $\Gamma_s$. By Corollary~\ref{cor-edges}, the edges of $\Gamma_s$ are described in Table~\ref{table-edges}. As the collineation $\bar\phi$ maps lines (respectively 2-secants) of $\Q_n$ to lines (2-secants) of $\Q_n$, the map $\phi$ preserves adjacencies and non-adjacencies of vertex pairs of $\Gamma_s$. That is $\phi\in\Aut(\Gamma_s)$, and so 
 $(\Aut\Gamma)_{\alpha_s}\subseteq \Aut(\Gamma_s)$. 
 
Further, by Lemma~\ref{lem:three-orbs}, 
$\PGO(n,2)_{\alpha_s}$ is transitive on the points of $\Q_n$ of each type, so $(\Aut\Gamma)_{\alpha_s}$ is  transitive on the vertices of $\Gamma$ of each type. 
Hence 
$\Aut(\Gamma_s)$ has at most three orbits on the vertices of $\Gamma_s$. By
 Lemma~\ref{lem:diffcliq}, $\Aut(\Gamma_s)$ has at least three orbits on the vertices of $\Gamma_s$. Hence $\Aut(\Gamma_s)$ has exactly three orbits on the vertices of $\Gamma_s$, namely the vertices of each type.
 \end{proof}

\begin{lemma}\Label{lem:new}
For $0\leq s<g$,  
$\Aut(\Gamma_s)\subseteq\Aut\Gamma$.
\end{lemma}

\begin{proof}
First note that if $s=0$,  then by Theorem~\ref{thm:s0iso}, $\Gamma_0=\Gamma$  so $\Aut(\Gamma_0)=\Aut\Gamma$. 
Suppose $s>0$, and let $\phi\in\Aut(\Gamma_s)$. 
As $\Gamma$ and $\Gamma_s$ have the same set of vertices, $\phi$ is a bijection acting on the vertices of $\Gamma$. We show that $\phi$ preserves adjacencies and non-adjacencies of vertices in $\Gamma$. 

By Corollary~\ref{cor-edges} and Table~\ref{table-edges}, the only difference in adjacencies between vertices in $\Gamma$ and $\Gamma_s$ are between a vertex  of type (ii) and a vertex of type (iii). Let $X,Y$ be two vertices of $\Gamma$, there are two cases to consider. Firstly, if the pair $X,Y$ consists of one vertex of type (ii) and one vertex of type (iii), then $X,Y$ are adjacent in $\Gamma$ if and only if $X,Y$ are non-adjacent in $\Gamma_s$. Secondly, if the pair $X,Y$ does not consist of one vertex of type (ii) and one vertex of type (iii), then $X,Y$ are adjacent in $\Gamma$ if and only if $X,Y$ are adjacent in $\Gamma_s$. In either case, as $\phi$ preserves adjacency and non-adjacency in $\Gamma_s$, $\phi$  preserves the adjacency or non-adjacency of the vertex pair $X,Y$ in $\Gamma$. Hence $\phi\in\Aut\Gamma$ as required. 
\end{proof}

{\bf Proof of Theorem~\ref{thm-gp}}\ \ 
 By Lemma~\ref{lem:new}, $\Aut(\Gamma_s)\subseteq\Aut \Gamma$, and so $(\Aut (\Gamma_s))_{\alpha_s}\subseteq(\Aut\Gamma)_{\alpha_s}$. 
  As $s>0$,  
  by Lemma~\ref{thm:3-orbit}, $\Aut(\Gamma_s)$ fixes the set of vertices of type (i), that is $(\Aut (\Gamma_s))_{\alpha_s}=\Aut(\Gamma_s)$, hence $\Aut(\Gamma_s)\subseteq(\Aut\Gamma)_{\alpha_s}$. By Lemma~\ref{thm:3-orbit}, $(\Aut\Gamma)_{\alpha_s}\subseteq\Aut(\Gamma_s)$, hence $(\Aut\Gamma)_{\alpha_s}=\Aut(\Gamma_s)$ as required.
\hfill$\square$

Finally we show that given a graph $\Gamma_s$, we can reconstruct the graph $\Gamma$ and the quadric $\Q_n$.
If $s=0$ then $\Gamma=\Gamma_0$ by Theorem~\ref{thm:s0iso}. So suppose $0<s<g$, and define a graph $\Gamma$ whose vertices are the vertices of $\Gamma_s$.  The proof of  Lemma~\ref{lem:diffcliq} shows that the number of maximal cliques through a vertex of $\Gamma_s$ of type (i) is greater than 
 the number of maximal cliques through a vertex of type (ii), which  is greater than 
  the number of maximal cliques through a vertex  of type (iii). Hence we can partition the vertices of $\Gamma_s$ into their types by using 
 the number of maximal cliques through them.  
 Define the edges of $\Gamma$ to be the same as the edges of $\Gamma_s$, except swapping the adjacencies between vertices  of type (ii) and (iii). Then by Corollary~\ref{cor-edges}, $\Gamma$ is the point-graph of the quadric $\Q_n$ used to construct $\Gamma_s$. 
We can now reconstruct the quadric $\Q_n$ from $\Gamma$ as follows. 
The maximal cliques of $\Gamma$ are exactly the generators of $\Q_n$ in $\PG(n,2)$.  By intersecting the generators of $\Q_n$, we can  recover firstly the $(g-1)$-spaces of $\Q_n$, and so on, constructing the lattice of subspaces of the generators. Hence we can construct the points of $\Q_n$, all the lines contained in $\Q_n$, the planes contained in $\Q_n$, \dots, the $g$-spaces contained in $\Q_n$.

\section{Conclusion}

In summary, Table~\ref{table3} lists the parameters of the strongly regular graphs arising from the point-graph of each type of non-singular quadric. Further, we list the number of new non-isomorphic graphs with these parameters arising from our construction (that is, not including $\Gamma_0=\Gamma$). 
%
%

\begin{table}[h]\caption{Parameters of the strongly regular graphs $\Gamma_s$, $0\leq s<g$}\label{table3}
\begin{center}\begin{tabular}{|p{3.1 cm}|c|c|c|c|p{1.8cm}|}
\hline
\centering{quadric} &$\E_{2r+1}$, $r\geq 2$ &  $\H_{2r+1}$, $r\geq 1$ &  $\P_{2r}$, $r\geq 2$\\
\hline
\centering{$v$} &$2^{2r+1}-2^r-1$&$2^{2r+1}+2^r-1$&$2^{2r}-1$     \\
\centering{$k$}&$2^{2r}-2^r-2$&$2^{2r}+2^r-2$&$2^{2r-1}-2$\\
\centering{$\lambda$}&$2^{2r-1}-2^{r}-3$&$2^{2r-1}+2^{r}-3$&$2^{2r-2}-3$\\
\centering{$\mu$}&$2^{2r-1}-2^{r-1}-1$&$2^{2r-1}+2^{r-1}-1$&$2^{2r-2}-1$\\
\hline
\centering{\small number of new non-isomorphic  graphs}
&$r-2$&$r-1$&$r-2$\\ 
\hline
\end{tabular}\end{center}

\end{table}%

%
%
%

\end{document}